\documentclass[11pt]{article}
\usepackage{amscd, amsmath, amssymb, amsthm}
\usepackage[all,cmtip]{xy}
\usepackage[pagebackref]{hyperref}

\title{Klt varieties with conjecturally minimal volume  }
\author{Burt Totaro}
\date{  }

\def\N{\text{\bf N}}

\def\Q{\text{\bf Q}}
\def\R{\text{\bf R}}
\def\C{\text{\bf C}}
\def\P{\text{\bf P}}

\DeclareMathOperator{\vol}{vol}
\DeclareMathOperator{\mult}{mult}
\DeclareMathOperator{\discrep}{discrep}
\DeclareMathOperator{\lct}{lct}
\DeclareMathOperator{\glct}{glct}
\DeclareMathOperator{\sing}{sing}

\begin{document}
\maketitle
\newtheorem{theorem}{Theorem}[section]
\newtheorem{corollary}[theorem]{Corollary}
\newtheorem{lemma}[theorem]{Lemma}

\theoremstyle{definition}
\newtheorem{definition}[theorem]{Definition}
\newtheorem{example}[theorem]{Example}
\newtheorem{conjecture}[theorem]{Conjecture}
\newtheorem{question}[theorem]{Question}

\theoremstyle{remark}
\newtheorem{remark}[theorem]{Remark}

We construct mildly singular (klt) complex projective varieties
with ample canonical class and the smallest known volume.
We also find exceptional klt Fano varieties with the smallest known
volume. In fact,
we conjecture that our examples
have the minimum volume in every dimension $n$,
and we give low-dimensional evidence to support this. Crudely, the volume
is about $1/2^{2^n}$.
These varieties improve on the examples
by Chengxi Wang and me \cite[Theorem 0.1]{TW}. Those examples had
roughly the right asymptotics, but they were known not
to be optimal.

By definition, the {\it volume }of a normal projective variety $X$
measures the asymptotic growth of the plurigenera,
$$\vol(X):=\lim_{m\to\infty} h^0(X,mK_X)/(m^n/n!).$$
This is equal to the intersection number $K_X^n$
if the canonical class $K_X$ is ample; it is the basic discrete
invariant for a variety of general type, analogous
to the genus of a curve. (When $X$ is klt and $K_X$ is ample,
this is literally
the volume of $X$ with its unique K\"ahler-Einstein metric of Ricci curvature
$-1$, up to a constant factor \cite[Theorem C]{EGZ}.)
By Hacon-M\textsuperscript{c}Kernan-Xu, there is a positive
lower bound for the volumes of all klt complex projective varieties
with ample canonical class, depending only on the dimension
\cite[Theorem 1.3]{HMXacc}. Finding an explicit bound
is a central problem in the classification of algebraic varieties,
wide open in dimensions
at least 3. (Alexeev and Liu gave a bound in dimension 2
\cite[section 10]{AL}.)
We are showing that the bound must go to zero
extremely fast, and conjecturally we give the exact bound
in each dimension.
In related work,
Koll\'ar found a klt {\it pair }with ample canonical
class and standard coefficients (described in section \ref{sectample})
that conjecturally has minimal volume
among such pairs
\cite{Kollarlog}, \cite[Introduction]{HMXbir}.
If we restrict to varieties with milder singularities,
Esser, Wang, and I constructed varieties in several classes
with small volume, such as smooth varieties of general type
or terminal Fano varieties \cite{ETW}.

The examples here are hypersurfaces in weighted projective spaces.
In contrast to most previous work, including
the examples by Wang and me \cite{TW},
the new hypersurfaces are not quasi-smooth. As a result, the
singularities are not always quotient singularities, and proving
that they are klt is more subtle. In dimension 2,
our example is Alexeev-Liu's klt surface with ample canonical
class and volume $1/48983$ \cite[Theorem 1.4]{AL}. Their construction
was different, but we find that their example
is in fact a non-quasi-smooth hypersurface, namely
$X_{438}\subset \P^3(219,146,61,11)$.

We also consider an analogous problem for klt Fano varieties.
The anticanonical volume of a klt Fano variety can be arbitrarily
small in a given dimension. (For example, for any positive integer $a$,
the weighted projective plane $Y=\P^2(2a+1,2a,2a-1)$
is a klt del Pezzo surface with $\vol(-K_Y)=18a/(4a^2-1)$.)
However, Birkar showed
that exceptional klt Fano varieties form a bounded family
in each dimension,
and in particular there is a positive lower bound for their
volumes \cite[Theorem 1.3]{Birkarcomp}. (By definition, a klt
Fano variety $X$ is {\it exceptional }if the pair $(X,D)$ is klt
for every effective $\Q$-divisor $D$ that is $\Q$-linearly equivalent
to $-K_X$. Equivalently, the {\it global log canonical threshold }(or
{\it $\alpha$-invariant}) of $X$ is greater than 1. Non-exceptional Fano
varieties can be analyzed in terms of lower-dimensional Fano pairs,
and so exceptional Fano varieties can be considered
the core of the classification problem for Fano varieties.)
By the recent proof of the Yau-Tian-Donaldson conjecture
for singular varieties, every exceptional klt Fano variety
has a K\"ahler-Einstein metric \cite[Theorem 1.4]{OS},
\cite[Theorem 1.6]{LXZ}.

We construct here what we conjecture to be
the exceptional klt Fano variety of minimum volume
in every dimension $n$. Again, the volume is roughly $1/2^{2^n}$.
It seems that the only known examples
of exceptional klt Fano varieties in high dimensions are
the quasi-smooth hypersurfaces
found by Johnson and Koll\'ar
\cite[Proposition 3.3]{JK4}.
We extend their argument to prove exceptionality
of our examples. In fact, we compute the global log canonical threshold
exactly. Crudely, it is about $2^{2^n}$, hence greater than 1 as we want.
The method, based on weighted multiplicities
in place of the usual multiplicity of a variety at a point,
should be useful for many other Fano varieties.

We give low-dimensional evidence
that our exceptional Fano varieties have the minimum
volume in each dimension.
As in the case of ample
canonical class, seeking optimal examples among all exceptional
Fano varieties
forces us to consider
non-quasi-smooth hypersurfaces in weighted projective space.
In dimension 2, our example (apparently new)
is the exceptional klt del Pezzo surface $X_{354}\subset \P^3(177,
118, 49, 11)$, for which $\vol(-K_X)=1/31801$.
This is smaller than the volume of any exceptional klt del Pezzo surface
with Picard number 1, by Lacini's classification of those surfaces
\cite{Lacini,KM}. The surface here has Picard number 2.

Finally, in every even dimension,
we construct klt varieties with ample canonical class,
and also klt Fano varieties, with the largest known bottom weight
(Theorems \ref{fanoglct} and \ref{kampleglct}).
The {\it bottom weight }means the smallest positive integer $m$
such that $H^0(X,mK_X)\neq 0$ in the $K$-ample case, or
$H^0(X,-mK_X)\neq 0$ in the Fano case. The global log canonical
threshold should be extremely large in these examples,
perhaps even maximal (Question \ref{glctfano}).

This work was supported by NSF grant DMS-2054553.
Thanks to Louis Esser, Yuchen Liu, Joaqu\'in Moraga,
and Chengxi Wang for their suggestions. Liu spotted a misuse
of Blum-Liu-Xu's work \cite{BLX} in the first version of this paper.

\section{Background}

Our examples use {\it Sylvester's sequence}, defined by
$s_0=2$ and $s_{n+1}=s_n(s_n-1)+1$. The sequence begins
$2,3,7,43,1807,\ldots$. We have
$s_{n+1}=s_0\cdots s_n+1$, and hence the numbers in Sylvester's
sequence are pairwise coprime. More important for applications
to extremal problems is that the sums
of the reciprocals of these numbers converge quickly to 1:
$$\frac{1}{s_0}+\cdots+\frac{1}{s_{n-1}}=1-\frac{1}{s_n-1}.$$
Here $s_n$ grows doubly exponentially in $n$, and so this
sum is very close to 1. In fact, for every positive integer $n$,
Soundararajan showed
that the sum above is the closest to 1 of all sums
of $n$ unit fractions that are less than 1 \cite{Soundararajan}.

There is a constant $c\doteq 1.264$ such that $s_i$ is the closest
integer to $c^{2^{i+1}}$ for all $i\geq 0$ \cite[equations
2.87 and 2.89]{GK}. For example, it follows that $s_i>2^{2^{i-1}}$
for all $i\geq 0$. We write $a_i\sim b_i$ to mean that two sequences
of positive real numbers are asymptotic, meaning that $a_i/b_i$ converges
to 1 as $i$ goes to infinity.

We consider algebraic varieties
over the complex numbers, although some of
the paper would work in any characteristic.
A reference for the singularities of the minimal model program,
such as Kawamata log terminal (klt) and log canonical (lc),
is \cite{Kollarsings}. 
We often use without comment
that the klt or lc properties of a pair $(X,\Delta)$ are unchanged
under finite coverings $f$ of normal varieties
with $f$ \'etale in codimension 1
\cite[Corollary 2.43]{Kollarsings}. (This would not be true
for other singularity classes such as terminal or canonical.)
As a result, we do not need to distinguish between the klt or lc
property for a normal Deligne-Mumford stack
and for its associated coarse moduli space,
provided that the stabilizer groups are trivial in codimension 1.

For an effective
$\Q$-Cartier $\Q$-divisor $D$ on a klt variety $X$, the {\it log
canonical threshold} $\,\lct(X,D)$ is the supremum of the real numbers
$\lambda$ such that the pair $(X,\lambda D)$ is lc. For a klt Fano variety $X$,
the {\it global log canonical threshold }(or {\it $\alpha$-invariant})
$\glct(X)$ is the supremum of the real numbers $\lambda$ such that
$(X,\lambda D)$ is lc for every effective $\Q$-divisor $D$ with
$D\sim_{\Q}-K_X$. It follows that a Fano variety with global
log canonical threshold greater than 1 must be exceptional.
In fact, $\glct(X)>1$ is equivalent to exceptionality,
by Birkar \cite[Theorem 1.7]{Birkarsings}.

For positive integers $a_0,\ldots,a_n$, the weighted
projective space $Y=\P(a_0,\ldots,a_n)$ means the quotient
variety $(A^{n+1}-0)/G_m$ over $\C$, where the multiplicative group
$G_m$ acts by $t(x_0,\ldots,x_n)=(t^{a_0}x_0,\ldots,t^{a_n}x_n)$
\cite[section 6]{Iano-Fletcher}.
Starting in section \ref{estimating}, we switch to viewing
weighted projective space as the quotient stack
$\mathcal{Y}=[(A^{n+1}-0)/G_m]$. Here $\mathcal{Y}$
is a smooth Deligne-Mumford stack with
canonical class
$K_{\mathcal{Y}}=O_{\mathcal{Y}}(-\sum a_j)$.
We say that $Y$ is {\it well-formed }if
the stack $\mathcal{Y}$ has trivial stabilizer in codimension 1,
or equivalently if $\gcd(a_0,\ldots,\widehat{a_j},\ldots,a_n)=1$
for each $j$. In the well-formed case, the canonical class
of the variety $Y$ is given by the same formula
as for the stack.

Here $O(1)$ is a line bundle on the stack $\mathcal{Y}$.
On the variety $Y$, with $Y$ well-formed, $O_Y(1)$
is only the reflexive sheaf associated to a Weil divisor, in general;
the divisor class $O_Y(m)$ is Cartier if and only if $m$ is a multiple
of every weight $a_j$.
The intersection number $\int_{\mathcal{Y}} c_1(O(1))^n$
is $1/(a_0\cdots a_n)$.
More generally, for an integral closed substack $Z$ of dimension $r$
in $\mathcal{Y}$, its {\it degree }means $\int_Z c_1(O(1))^r$.

Let $Y$ be a well-formed weighted projective space. A closed subvariety
$X$ of $Y$ is called {\it quasi-smooth }if its affine cone in $A^{n+1}$
is smooth outside the origin. (Equivalently, the inverse image
of $X$ in the stack $\mathcal{Y}$ is smooth over $\C$.)
In particular,
a quasi-smooth subvariety has only cyclic quotient singularities
and hence is klt. Also, $X$
is {\it well-formed }if $Y$ is well-formed and the codimension
of $X\cap Y^{\sing}$ in $X$ is at least 2. (For a well-formed
weighted projective space $Y$,
the singular locus of the variety $Y$ corresponds to the locus
where the stack $\mathcal{Y}$ has nontrivial stabilizer.)

For a well-formed normal hypersurface $X$ of degree $d$ in a weighted
projective space $Y$, we have $K_X=O_X(d-\sum a_j)$. (We are not assuming
quasi-smoothness of $X$.) Indeed,
the canonical class of a normal variety is defined as a Weil divisor
up to linear equivalence, and so we are free to delete closed subsets
of codimension at least 2 from $X$ in order to prove this formula.
So we can delete the singular locus of $X$ and the singular locus
of $Y$ from $X$ and $Y$, and then $K_X=O_X(d-\sum a_j)$ is the usual
adjunction formula for a smooth hypersurface in a smooth variety.

Note an ambiguity
in the notion of ``degree'': if $X$ is a hypersurface of degree $d$
in $\mathcal{Y}$,
meaning that it is defined by a weighted-homogeneous polynomial
of degree $d$,
then its degree as a substack of $\mathcal{Y}$ is $X\cdot c_1(O(1))^{n-1}=
d/(a_0\cdots a_n)$.

\section{Klt varieties with ample canonical class}
\label{sectample}

\begin{theorem}
\label{ample}
For each integer $n$ at least 2, let 
$$a_{n+1}=\begin{cases} \frac{1}{4}(s_n^2-s_n+2)&\text{if $n$ is even}\\
\frac{1}{4}(s_n^2-3s_n+4)&\text{if $n$ is odd.}
\end{cases}$$
Let $a_n=(s_n-2)a_{n+1}+(s_n-1)$, $x=1+a_n+a_{n+1}$,
$d=(s_n-1)x$, and $a_i=d/s_i$ for $0\leq i\leq n-1$.
Then there is a hypersurface $X$ of degree $d$ in
$\P^{n+1}(a_0,\ldots,a_{n+1})$ that is well-formed and klt,
with $K_X=O_X(1)$. It has volume
$$\frac{1}{(s_n-1)^{n-2}x^{n-1}a_na_{n+1}},$$
which is asymptotic to $2^{2n+2}/s_n^{4n}$.
In particular, this is less
than $1/2^{2^n}$.
\end{theorem}

Explicitly, define $X$ by the equation, for $n\geq 2$ even:
$$0=x_0^2+x_1^3+\cdots
+x_{n-1}^{s_{n-1}}+x_n^{s_n}x_{n+1}+x_1\cdots x_nx_{n+1}^b,$$
where $b=(s_n^2-2s_n+7)/2$. For $n\geq 3$ odd, define $X$ by
$$0=x_0^2+x_1^3+\cdots
+x_{n-1}^{s_{n-1}}+x_n^{s_n}x_{n+1}+x_1\cdots x_{n-1}x_n^2x_{n+1}^b,$$
where now $b=(s_n^2-4s_n+11)/2$. Since the number of monomials is equal
to the number of variables, any linear combination
of these monomials with all coefficients nonzero
defines an isomorphic variety, by scaling
the variables. One can check
that the monomials shown are all the monomials of degree $d$,
and hence that an open subset
of all hypersurfaces of degree $d$ are isomorphic to the one indicated;
but we will not need those facts.

Note that $X$ is not quasi-smooth.

\begin{conjecture}
For each integer $n$ at least 2,
the variety in Theorem \ref{ample}
has the minimum volume among all klt projective $n$-folds with ample
canonical class.
\end{conjecture}

We know that there is some
positive lower bound for the volume in each dimension,
by Hacon-M\textsuperscript{c}Kernan-Xu \cite[Theorem 1.3]{HMXacc}.

In dimension 2, our example is 
$X_{438}\subset \P^3(219,146,61,11)$, with volume $1/48983
\doteq 2.0\times 10^{-5}$.
This is the smallest known volume for a klt surface with ample
canonical class. This example was found earlier by Alexeev and Liu,
without the description as a hypersurface
\cite[Theorem 1.4]{AL}. It has smaller volume
than all quasi-smooth hypersurfaces
of dimension 2 with $K_X=O_X(1)$, 
by Brown and Kasprzyk's
computer classification \cite{BK,GRD}. Namely, the best
quasi-smooth hypersurface is
$X_{316}\subset \P^3(158,85,61,11)$, with volume
$2/57035\doteq 3.5\times 10^{-5}$.

In dimension 3,
our example is 
$$X_{762090}\subset \P^4(381045,
254030,108870,17713,431),$$
which has volume about $9.5 \times 10^{-18}$.
This beats the best previously known 3-fold,
the quasi-smooth hypersurface
$$X_{340068}\subset \P^4(170034,113356,47269,9185,223),$$
which has volume about $1.8\times 10^{-16}$.
(The latter example is optimal among quasi-smooth
hypersurfaces with $K_X=O_X(1)$, by Brown and Kasprzyk's program.
It is part of the sequence of examples constructed
by Wang and me \cite[section 2]{TW}.)
Finally, in dimension 4, our new example has volume
about $8.0\times 10^{-50}$. Again, this beats the optimal
quasi-smooth hypersurface with $K_X=O_X(1)$ in dimension 4,
which has volume about $1.4\times 10^{-47}$
\cite[ID 538926]{GRD}.

Finally, our klt variety with ample canonical class
has volume quite close to that of Koll\'ar's conjecturally
optimal example in the broader setting of klt pairs
with ample canonical class and standard coefficients (meaning
coefficients of the form $(m-1)/m$ for positive integers $m$).
That example is
$$(Y,\Delta)=\bigg( \P^n,\frac{1}{2}H_0+\frac{2}{3}H_1
+\frac{6}{7}H_2+\cdots+\frac{s_{n+1}-1}{s_{n+1}}H_{n+1}\bigg),$$
where $H_0,H_1,\ldots,H_{n+1}$ are $n+2$ general hyperplanes.
The volume of $K_Y+\Delta$ is $1/(s_{n+2}-1)^n$, which is (crudely)
about $1/2^{2^n}$. Our example in Theorem \ref{ample} has
$\vol(X)/\vol(K_Y+\Delta)$ about $2^{2n+2}$. (Precisely,
$\log(\vol(X)/\vol(K_Y+\Delta))$ is asymptotic
to $(2n+2)\log 2$ as $n$ goes to infinity.)
So $\vol(X)$ is bigger than $\vol(K_Y+\Delta)$,
but not by much, since $2^{2n+2}$ is far smaller than $2^{2^n}$.
That is some further evidence for the optimality of Theorem \ref{ample}.

\begin{proof}
(Theorem \ref{ample})
To explain the choice of weights $a_i$, we first prove some properties
of these numbers. First, we have $d-\sum a_i=1$ (which will imply
that $K_X=O_X(1)$), because $d-\sum_{i=0}^{n-1} a_i
=d(1-1/s_0-\cdots -1/s_{n-1})=d/(s_n-1)=x=1+a_n+a_{n+1}$.
Next, let us check that the monomials listed in the equation
of $X$ (above) have degree $d$. For $n$ even,
let $b=(s_n^2-2s_n+7)/2$; then we have to show that
$d=2a_0=3a_1=\cdots=s_{n-1}a_{n-1}=s_na_n+a_{n+1}
=a_1+\cdots+a_n+ba_{n+1}$. All these equations except the last one
are easy by our choice of weights. For the last one,
note that 
\begin{align*}
d-a_1-\cdots-a_n&=d\bigg(1-\frac{1}{s_1}-\cdots-\frac{1}{s_{n-1}}\bigg)-a_n\\
&=d\bigg(\frac{1}{2}+\frac{1}{s_n-1}\bigg)-a_n\\
&=\frac{s_n+1}{2}x-a_n\\
&=\frac{s_n+1}{2}+\frac{s_n-1}{2}a_n+\frac{s_n+1}{2}a_{n+1}\\
&=\frac{1}{2}(s_n^2-s_n+2)+\frac{1}{2}(s_n^2-2s_n+3)a_{n+1}\\
&=\frac{1}{2}(s_n^2-2s_n+7)a_{n+1},
\end{align*}
as we want. For $n$ odd, a similar calculation shows
that the monomials in the equation of $X$ have degree $d$.

We first show that the weighted projective
space $Y=\P^{n+1}(a_0,\ldots,a_{n+1})$ is well-formed.
That is, we have to show that $\gcd(a_0\ldots,\widehat{a_j},
\ldots,a_{n+1})=1$ for each $j$. It suffices to show
that $\gcd(a_{n+1},x)=1$, $\gcd(a_n,x)=1$,
and $\gcd(a_{n+1},a_n,s_n-1)=1$. (This uses that $s_n-1=
s_0\cdots s_{n-1}$, where $s_0,\ldots,s_{n-1}$ are pairwise
coprime.)

First, note
that $s_n$ is $7\pmod{8}$ if $n\geq 2$ is even and $3\pmod{8}$
if $n\geq 3$ is odd. This is immediate by induction from
the recurrence $s_{n+1}=s_n(s_n-1)+1$. It follows that
$s_n^2-s_n+2$ is $4\pmod{8}$ if $n\geq 2$ is even,
and that $s_n^2-3s_n+4$ is $4\pmod{8}$ if $n\geq 3$ is odd.
So $a_{n+1}$ is odd in both cases.

Next, let us show that $\gcd(a_{n+1},x)=1$.
Suppose that a prime number $p$ divides both $a_{n+1}$ and $x$.
Since $a_{n+1}$ is odd, $p$ is not 2.
Since $x=1+a_n+a_{n+1}$, we have $a_n\equiv -1\pmod{p}$.
Since $a_n=(s_n-2)a_{n+1}+(s_n-1)$, we have $-1=s_n-1\pmod{p}$,
so $p$ divides $s_n$. If $n\geq 2$ is even, 
$s_n^2-s_n+2\equiv 2\pmod{p}$, and this is not zero
mod $p$. So $a_{n+1}$ is not zero mod $p$,
a contradiction. Likewise, for $n$ odd,
$s_n^2-3s_n+4\equiv 4\pmod {p}$, and this is not zero
mod $p$. So $a_{n+1}$ is not zero mod $p$,
a contradiction. Thus $a_{n+1}$ is prime to $x$. 

To show that $\gcd(a_n,x)=1$,
suppose that a prime number $p$ divides both $a_n$ and $x$.
Since $x=1+a_n+a_{n+1}$, we have $a_{n+1}\equiv -1\pmod{p}$.
Since $a_n=(s_n-2)a_{n+1}+(s_n-1)$, we have $0\equiv -(s_n-2)+(s_n-1)
\equiv -1\pmod{p}$, a contradiction. So $a_n$ is prime to $x$.

It remains to show that $\gcd(a_n,a_{n+1},s_n-1)=1$; in fact,
we show that $\gcd(a_{n+1},s_n-1)=1$.
Let $p$ be a prime number that divides $a_{n+1}$
and $s_n-1$. We have $s_n\equiv 1\pmod{p}$,
so $s_n^2-s_n+2\equiv 2\pmod{p}$
and $s_n^2-3s_n+4\equiv 2\pmod{p}$. Since $p$ is not 2,
these two expressions are not zero
mod $p$. It follows that $a_{n+1}$ is not zero mod $p$,
a contradiction. This completes the proof that $Y$
is well-formed.

To show that $X$ is well-formed, it remains to show
that $X$ does not contain any $(n-1)$-dimensional
coordinate linear subspace of $Y$
along which $Y$ is singular.
Since the equation of $X$ includes the monomials
$x_0^2,x_1^3,\ldots,x_{n-1}^{s_{n-1}}$, and also $x_n^{s_n}x_{n+1}$,
$X$ does not contain
any positive-dimensional coordinate linear subspace of $Y$. Since $n\geq 2$,
$X$ is well-formed.

Next, we show that $X$ is klt. First suppose that $n$ is even.
In this case, the equation defining $X$ is
$0=x_0^2+x_1^3+\cdots
+x_{n-1}^{s_{n-1}}+x_n^{s_n}x_{n+1}+x_1\cdots x_nx_{n+1}^b$,
where $b=(s_n^2-2s_n+7)/2$. In particular,
the base locus of the linear system $|O(d)|$ on $Y$
is contained in the two last coordinate points,
$[0,\ldots,0,1,0]$ and $[0,\ldots,0,1]$. 
Since we are in characteristic zero,
Bertini's theorem on $A^{n+2}-0$ gives that
$X$ is quasi-smooth outside those two points. In view
of the monomial $x_n^{s_n}x_{n+1}$, $X$ is also quasi-smooth
at the point $[0,\ldots,0,1,0]$. So $X$ is klt outside
the point $[0,\ldots,0,1]$.

At the point $[0,\ldots,0,1]$, $X$ is not quasi-smooth,
but we will show that it is still klt. Because the number
of monomials in the equation of $X$ is equal to the number
of variables, a general linear combination of those monomials
is isomorphic to $X$, by scaling the variables. So it suffices
to show that the hypersurface $X'$ defined by such a general
linear combination is klt at $[0,\ldots,0,1]$.
In coordinates $x_{n+1}=1$, the equation of $X'$
is $0=c_0x_0^2+c_1x_1^3+\cdots
+c_{n-1}x_{n-1}^{s_{n-1}}+c_nx_n^{s_n}+c_{n+1}x_1\cdots x_n$
for general complex numbers $c_i$.
The open subset $x_{n+1}\neq 0$ of $X'$
is the quotient by the finite cyclic group
$\mu_{a_{n+1}}$ of the hypersurface
with the same equation
in $A^{n+1}$. Because the klt property is preserved
by finite quotients, it suffices to show that such a general
hypersurface $S$ in $A^{n+1}$ has canonical singularities
(or equivalently, rational singularities).

Ishii and Prokhorov (following earlier work)
described when the general hypersurface $S\subset A^{n+1}$ with equation
spanned by a given set $I$ of monomials has canonical singularities,
as follows \cite[Proposition 2.9]{IP}.
By definition, the {\it Newton
polyhedron }of a finite subset $I\subset \R^{n+1}$
is the convex hull of $I$ plus the positive
orthant, $(\R^{\geq 0})^{n+1}$.

\begin{theorem}
\label{newton}
Let $I$ be a finite subset of $\N^{n+1}$, viewed as monomials
in $\C[x_0,\ldots,x_n]$. Let $S\subset A^{n+1}_{\C}$
be the zero set of a general linear combination
of these monomials. Assume that for every $0\leq i<j\leq n$,
there is an monomial in $I$ that contains neither $x_i$ nor $x_j$;
then the hypersurface $S$ is normal.
If the Newton polyhedron of $I$ in $\R^{n+1}$ contains
$(1,\ldots,1)$ in its interior, then the hypersurface
$S$ has canonical singularities.
The converse holds if $I$ contains no monomial of degree 1.
\end{theorem}

As above, for $n\geq 2$ even, let $S$ be the zero set in $A^{n+1}$
of a general linear combination of $x_0^2,x_1^3,\ldots,
x_{n-1}^{s_{n-1}},x_n^{s_n}$, and $x_1\cdots x_n$.
The first condition in Theorem \ref{newton} (ensuring normality of $S$)
is clear. Therefore, to show that $S$ is canonical and hence $X$ is klt,
it suffices to show that the convex hull of the points
$(2,0,\ldots,0)$, $(0,3,0,\ldots,0)$, \ldots,
$(0,\ldots,0,s_n)$, $(0,1,\ldots,1)$ in $\R^{n+1}$ contains
a point with all coordinates less than 1. In fact, we only need
three of these points: namely, $(5/12)(2,0,\ldots,0)+
(1/6)(0,3,0\ldots,0)+(5/12)(0,1,\ldots,1)$ has all coordinates
less than 1. Thus $X$ is klt when its dimension $n$ is even.

For $n\geq 3$ odd, the equation defining $X$ is
$0=x_0^2+x_1^3+\cdots
+x_{n-1}^{s_{n-1}}+x_n^{s_n}x_{n+1}+x_1\cdots x_{n-1}x_n^2x_{n+1}^b$,
where now $b=(s_n^2-4s_n+11)/2$. As above, it is equivalent
to consider a general linear combination of these monomials.
As in the case of $n$
even, $X$ is quasi-smooth outside the point
$[0,\ldots,0,1]$. To show that $X$ is klt at that point,
it suffices to show that the convex hull of the points
$(2,0,\ldots,0)$, $(0,3,0,\ldots,0)$, \ldots,
$(0,\ldots,0,s_n)$, $(0,1,\ldots,1,2)$ in $\R^{n+1}$ contains
a point with all coordinates less than 1. Again, we only need
three of these points: namely, $(5/12)(2,0,\ldots,0)+
(1/6)(0,3,0\ldots,0)+(5/12)(0,1,\ldots,1,2)$ has all coordinates
less than 1. Thus $X$ is klt whether $n$ is even or odd.

Since $X$ is well-formed, the adjunction formula holds,
meaning that $K_X=O_X(d-\sum a_j)=O_X(1)$.
Therefore, 
\begin{align*}
\vol(K_X)&=\frac{d}{a_0\cdots a_{n+1}}\\
&= \frac{1}{(s_n-1)^{n-2}x^{n-1}a_na_{n+1}}.
\end{align*}
Here $a_{n+1}\sim s_n^2/4$
and $a_n\sim s_n^3/4$ (much bigger than $a_{n+1}$),
so $x\sim s_n^3/4$.
It follows that $\vol(K_X)\sim 2^{2n+2}/s_n^{4n}$.
\end{proof}

\section{Klt Fano varieties}

\begin{theorem}
\label{fano}
For each integer $n$ at least 2, let 
$$a_{n+1}=\begin{cases} \frac{1}{4}(s_n^2-s_n+2)&\text{if $n$ is even}\\
\frac{1}{4}(s_n^2-3s_n+4)&\text{if $n$ is odd.}
\end{cases}$$
Let $a_n=(s_n-2)a_{n+1}-(s_n-1)$, $x=-1+a_n+a_{n+1}$,
$d=(s_n-1)x$, and $a_i=d/s_i$ for $0\leq i\leq n-1$.
Then there is a hypersurface $X$ of degree $d$ in
$\P^{n+1}(a_0,\ldots,a_{n+1})$ that is a well-formed klt Fano variety,
with $-K_X=O_X(1)$. The volume of $-K_X$ is
$$\frac{1}{(s_n-1)^{n-2}x^{n-1}a_na_{n+1}},$$
which is asymptotic to $2^{2n+2}/s_n^{4n}$.
In particular, this is less
than $1/2^{2^n}$.
\end{theorem}

Explicitly, define $X$ by the equation, for $n\geq 2$ even:
$$0=x_0^2+x_1^3+\cdots
+x_{n-1}^{s_{n-1}}+x_n^{s_n}x_{n+1}+x_1\cdots x_nx_{n+1}^b,$$
where $b=(s_n^2-2s_n-1)/2$. For $n\geq 3$ odd, define $X$ by
$$0=x_0^2+x_1^3+\cdots
+x_{n-1}^{s_{n-1}}+x_n^{s_n}x_{n+1}+x_1\cdots x_{n-1}x_n^2x_{n+1}^b,$$
where now $b=(s_n^2-4s_n+3)/2$. Since the number of monomials is equal
to the number of variables, any linear combination
of these monomials with all coefficients nonzero
defines an isomorphic variety, by scaling
the variables. One can check
that the monomials shown are all the monomials of degree $d$,
and hence that an open subset
of all hypersurfaces of degree $d$ are isomorphic to the one indicated;
but we will not need those facts.

Note that $X$ is not quasi-smooth. We show in Theorems
\ref{evenex} and \ref{oddex} that this Fano variety is exceptional.

\begin{conjecture}
For each integer $n$ at least 2,
the variety in Theorem \ref{fano}
has the minimum anticanonical volume among all exceptional
klt Fano $n$-folds.
\end{conjecture}

Birkar showed that exceptional klt Fano varieties form
a bounded family in each dimension, and so there is some positive
lower bound for their volumes
\cite[Theorem 1.3]{Birkarcomp}.

In dimension 2, our example
is the klt del Pezzo surface $X_{354}\subset \P^3(177,
118, 49, 11)$, for which $\vol(-K_X)=1/31801\doteq 3.1\times 10^{-5}$.
The lowest volume previously known for an exceptional
klt del Pezzo surface occurs for Johnson-Koll\'ar's quasi-smooth surface
$X_{256}\subset \P^3(128,69,49,11)$,
with volume $2/37191\doteq 5.4\times 10^{-5}$
\cite[Theorem 8]{JKEinstein}. Exceptionality of the latter surface
follows from Johnson-Koll\'ar's theorem that a well-formed quasi-smooth
hypersurface $X_d\subset \P^{n+1}(a_0,\ldots,a_{n+1})$
with $d=-1+\sum a_j$ (so $K_X=O_X(-1)$) 
and $a_0\geq\cdots\geq a_{n+1}$ is exceptional
if $d\leq a_na_{n+1}$ \cite[Proposition 3.3]{JK4}.

In dimension 3, our example is the klt Fano 3-fold
$$X_{758478}\subset\P^4(379239,252826,108354,17629,431),$$
with anticanonical volume about $9.6\times 10^{-18}$.
The lowest previously known volume
of an exceptional klt Fano 3-fold
occurs for Johnson-Koll\'ar's
quasi-smooth hypersurface
$$X_{336960}\subset
\P^4(168480,112320,46837,9101,223),$$
which has volume
about $1.9\times 10^{-16}$ \cite[introduction]{JK4}.
Finally, our klt Fano 4-fold has volume about
$8.0\times 10^{-50}$. The smallest previously known volume
of an exceptional klt Fano 4-fold is
$1.4\times 10^{-47}$, again for a certain quasi-smooth
hypersurface \cite[ID 1233322]{GRD}.

\begin{proof}
(Theorem \ref{fano}) The proof is similar to that of
Theorem \ref{ample}, where the canonical class is ample.
In particular, the weight $a_{n+1}$ is the same
in the two theorems,
and the formulas for $a_n$ and $x$ differ only by sign changes.
Modifying the calculation at the start of the proof
of Theorem \ref{ample} by these sign changes shows
that $d-\sum a_i=-1$ (rather than 1).
Also, whether $n$ is even or odd,
we compute that the monomials in the equation
for $X$ have degree $d$.

As shown in the proof of Theorem \ref{ample},
$a_{n+1}$ is odd.
The sign changes make no difference to the proof
that $\gcd(a_{n+1},x)=1$, $\gcd(a_n,x)=1$,
and $\gcd(a_{n+1},a_n,s_n-1)=1$.
Therefore,
$Y=\P^{n+1}(a_0,\ldots,a_{n+1})$ is well-formed.

To show that $X$ is well-formed, it remains to show
that $X$ does not contain any $(n-1)$-dimensional
coordinate linear subspace of $Y$
along which $Y$ is singular.
Since the equation of $X$ includes the monomials
$x_0^2,x_1^3,\ldots,x_{n-1}^{s_{n-1}}$, and $x_n^{s_n}x_{n+1}$,
$X$ does not contain
any positive-dimensional coordinate linear subspace of $Y$. Since $n\geq 2$,
$X$ is well-formed.

Next, we show that $X$ is klt. First suppose that $n$ is even.
In this case, the equation defining $X$ is
$0=x_0^2+x_1^3+\cdots
+x_{n-1}^{s_{n-1}}+x_n^{s_n}x_{n+1}+x_1\cdots x_nx_{n+1}^b$,
where $b=(s_n^2-2s_n-1)/2$. It follows that $X$ is
quasi-smooth and hence klt outside
the point $[0,\ldots,0,1]$.
At the point $[0,\ldots,0,1]$, $X$ is not quasi-smooth,
but we will show that it is still klt.
In coordinates $x_{n+1}=1$, the equation of $X$
is $0=x_0^2+x_1^3+\cdots
+x_{n-1}^{s_{n-1}}+x_n^{s_n}+x_1\cdots x_n$.
We showed in the proof of Theorem \ref{ample} that this
hypersurface in $A^{n+1}$
has canonical singularities near the origin.
Therefore, $X$ (the quotient by $\mu_{a_{n+1}}$) is klt
at $[0,\ldots,0,1]$, as we want.

For $n\geq 3$ odd, the equation defining $X$ is
$0=x_0^2+x_1^3+\cdots
+x_{n-1}^{s_{n-1}}+x_n^{s_n}x_{n+1}+x_1\cdots x_{n-1}x_n^2x_{n+1}^b$,
where now $b=(s_n^2-4s_n+3)/2$. Again, $X$ is quasi-smooth
and hence klt outside the last coordinate point
$[0,\ldots,0,1]$. To show that $X$ is klt at that point,
use coordinates $x_{n+1}=1$ to write the equation of $X$
as $0=x_0^2+x_1^3+\cdots+x_{n-1}^{s_{n-1}}+x_n^{s_n}+
x_1\cdots x_{n-1}x_n^2$. We showed 
in the proof of Theorem \ref{ample} that this
hypersurface in $A^{n+1}$
has canonical singularities near the origin.
Therefore, $X$ (the quotient by $\mu_{a_{n+1}}$) is klt
at $[0,\ldots,0,1]$, as we want, whether $n$ is even or odd.

Since $X$ is well-formed, the adjunction formula holds,
meaning that $K_X=O_X(d-\sum a_j)=O_X(-1)$.
Therefore, 
\begin{align*}
\vol(-K_X)&=\frac{d}{a_0\cdots a_{n+1}}\\
&= \frac{1}{(s_n-1)^{n-2}x^{n-1}a_na_{n+1}}.
\end{align*}
Here $a_{n+1}\sim s_n^2/4$
and $a_n\sim s_n^3/4$ (much bigger than $a_{n+1}$),
so $x\sim s_n^3/4$.
It follows that $\vol(-K_X)\sim 2^{2n+2}/s_n^{4n}$.
\end{proof}

\section{Estimating the log canonical threshold in terms of the weighted
tangent cone}
\label{estimating}

Sections \ref{estimating} to \ref{oddexsect}
will show that the klt Fano varieties
in Theorem \ref{fano} are exceptional (Theorems \ref{evenex}
and \ref{oddex}). This follows from
their global log canonical threshold (or $\alpha$-invariant)
being greater than 1. In fact, we compute the global log canonical
threshold exactly. The method should be useful for many other examples.

\begin{lemma}
\label{lc}
Let $a_0,\ldots,a_n$ be positive integers. Let $X$ be a hypersurface
in $A^{n+1}$ over $\C$ that contains the origin.
Let $\Delta$ be an effective $\Q$-Cartier $\Q$-divisor in $X$.
Let $X_1$ be the $a$-weighted tangent cone of $X$ at the origin; thus $X_1$
is a hypersurface of some degree $d$ in the weighted projective space
$Y=\P(a_0,\ldots,a_n)$, viewed as a smooth Deligne-Mumford stack over $\C$.
Likewise, let $\Delta_1\subset X_1$
be the weighted tangent cone of $\Delta$. Then $K_{X_1}+\Delta_1
\sim_{\Q} O_{X_1}(r)$ for some $r\in\Q$. If $r\leq 0$
and $(X_1,\Delta_1)$
is log canonical, then $(X,\Delta)$ is log canonical near 0.
\end{lemma}

This generalizes Koll\'ar's description of which cones
are lc, which (in effect) concerns the case $a_0=\cdots=a_n=1$
\cite[Lemma 3.1]{Kollarsings}.
We do not assume that $X$ is a weighted cone.
The proof is fairly straightforward once one is willing to use
stack-theoretic weighted blow-ups. In retrospect, they are
the right tool for the job.

\begin{proof}
Let $p\colon B\to X$ be the stack-theoretic weighted blow-up of $X\subset A^{n+1}$
at the origin
with the given weights $a=(a_0,\ldots,a_n)$. Explicit coordinate
charts can be found in \cite[section 3.4]{ATW}. The weighted
tangent cone $X_1$ is defined to be the exceptional divisor $E\subset B$;
in particular, $E$ is a hypersurface in the smooth stack
$Y=\P(a_0,\ldots,a_n)=[(A^{n+1}-0)/G_m]$. (Because we view $Y$ as a stack,
we can allow $a_0,\ldots,a_n$ to have a common factor;
that is, the weighted projective space $Y$ need not be well-formed.)

We are assuming that $(X_1,\Delta_1)$ is lc, so in particular
$E=X_1$ is normal. It is a Cartier divisor in the stack $B$.
By the adjunction formula, the canonical class $K_E=(K_B+E)|_E$
is given by $O_E(d-\sum_j a_j)$.
Since $\Delta$ is assumed to be $\Q$-Cartier on $X$,
and $K_X$ is Cartier since $X$ is a hypersurface, it follows that $K_X+\Delta$
is $\Q$-Cartier. Therefore, $K_E+\Delta_1\sim_{\Q} O(r)
\sim_{\Q}-rE|_E$
for some $r\in \Q$.
Write $\Delta_B$ for the birational transform $p^{-1}_*\Delta$.
Then $K_B+\Delta_B+(1+r)E\sim_{\Q}p^*(K_X+\Delta)$,
using that both sides are trivial on $E$.

By \cite[Definition 2.23]{Kollarsings}, it follows that the discrepancy
of $(X,\Delta)$ is given by
$$\discrep(X,\Delta)=\min(-1-r,\discrep(B,(1+r)E+\Delta_B)).$$
(Here $(X,\Delta)$ is lc if and only if $\discrep(X,\Delta)\geq -1$.)
So $(X,\Delta)$ is lc near 0 if $r\leq 0$ (as we assume)
and $(B,(1+r)E+\Delta_B)$ is lc
near $E$. Since $r\leq 0$, it suffices to show that $(B,E+\Delta_B)$
is lc. Since $E$ is a normal Cartier divisor in $B$, inversion
of adjunction says that this follows from $(E,\Delta_B|_E)=
(X_1,\Delta_1)$ being lc \cite[Proposition 4.5, Theorem 4.9]{Kollarsings}.
\end{proof}

The following corollary applies Lemma \ref{lc} to the case
of a ``weighted ordinary'' hypersurface singularity,
meaning that the weighted tangent cone is smooth.
That covers many examples. The singularities in this paper
are more complicated, however, and so we will have to go back
to Lemma \ref{lc}. Corollary \ref{weightedmult}
is a weighted version of an {\it Izumi-type inequality},
meaning a bound of the form
$\lct_0(X,D)\geq c_X/\mult(D)$ (cf.\ \cite[Theorem 3.2,
Remark 3.3]{Li}).

\begin{corollary}
\label{weightedmult}
Let $a_0\geq\cdots\geq a_n$ be positive integers. Let $X$ be a hypersurface
in $A^{n+1}$ over $\C$ that contains the origin.
Let $D$ be an effective $\Q$-Cartier $\Q$-divisor in $X$.
Let $X_1$ be the weighted tangent cone of $X$ at the origin; thus $X_1$
is a hypersurface of some degree $d$ in the weighted projective space
$Y=\P(a_0,\ldots,a_n)$, viewed as a smooth Deligne-Mumford stack over $\C$.
Suppose that the stack $X_1$ is smooth over $\C$.
Let $D_1\subset X_1$
be the weighted tangent cone of $D$. Let $b=-d+\sum_j a_j$.
If $b>0$,
then the log canonical threshold of $(X,D)$ near the origin satisfies:
$$\lct_0(X,D)\geq \frac{\min(a_{n-1}a_n,bd)}{a_0\cdots a_n\mult_a(D)}.$$
\end{corollary}

Here, given positive integers $a=(a_0,\ldots,a_n)$, the
{\it weighted multiplicity} $\mult_a(D)$ means $\deg(D_1)$,
the degree of the weighted tangent cone $D_1$
as a codimension-2 cycle in $Y$.

\begin{proof} (Corollary \ref{weightedmult})
Let $c=\min(a_{n-1}a_n,bd)/(a_0\cdots a_n\deg(D_1))$.
We want to show that $(X,cD)$ is lc near 0. Let $\Delta=cD$ 
and $\Delta_1=cD_1$ its weighted tangent cone.
By Lemma \ref{lc}, it suffices to show
that $K_{X_1}+\Delta_1
\sim_{\Q} O_{X_1}(r)$ with $r\leq 0$ and that
$(X_1,\Delta_1)$ is lc.

By the adjunction formula, we have $-K_{X_1}=O_{X_1}(b)$.
To show that $r\leq 0$, it is equivalent to show that
$\deg(\Delta_1)=\Delta_1\cdot c_1(O(1))^{n-2}
\leq (-K_{X_1})\cdot c_1(O(1))^{n-2}$,
where the intersection numbers are computed
on the $(n-1)$-dimensional stack $X_1\subset Y$.
So $(-K_{X_1})\cdot c_1(O(1))^{n-2}=bd/(a_0\cdots a_n)$.
Using this plus the fact that
$\Delta_1=cD_1$, the inequality above holds if
$c\leq bd/(a_0\cdots a_n\deg(D_1))$.
That holds by our definition of $c$.

It remains to show that $(X_1,\Delta_1)$ is lc.
We define the multiplicity at a point
of a irreducible closed substack (or an effective algebraic cycle)
in weighted projective space $Y$ to be the multiplicity at a corresponding
point of its inverse image in any orbifold chart $A^n\to [A^n/\mu_{a_i}]
\cong \{ x_i\neq 0\}\subset Y$. (This is independent of $i$,
because the different orbifold charts are \'etale-locally
isomorphic.) Johnson and Koll\'ar proved the following bound
\cite[Proposition 11]{JKEinstein}. (The last sentence of Theorem
\ref{jkbound} is not stated in their paper, but it is immediate
from their argument.)

\begin{theorem}
\label{jkbound}
Let $a_0\geq\cdots\geq a_n$ be positive integers. Let $M$
be an irreducible closed substack of dimension $r$ in the stack
$\P^n(a_0,\ldots,a_n)$. Then the multiplicity of $M$
at every point is at most $(a_0\cdots a_r)\deg(M)$.
If $M$ is not contained in the hyperplane $x_n=0$,
then this bound can be improved to $(a_0\cdots a_{r-1}a_n)\deg(M)$.
\end{theorem}

Returning to the proof of Corollary \ref{weightedmult},
Theorem \ref{jkbound} gives that
the $(n-2)$-dimensional cycle $D_1$ in $Y$
has multiplicity at every point at most
$(a_0\cdots a_{n-2})\deg(D_1)$.

Now use the assumption
that $X_1$ is a smooth stack. Then, for $0\leq i\leq n$,
the inverse image $X_{1,i}$ of $X_1$ in the $i$th orbifold chart
is a smooth hypersurface in $A^n$. For a smooth variety $S$ with an effective
$\Q$-divisor $T$, the pair $(S,T)$ is lc if $T$ has multiplicity at most 1
at every point \cite[Claim 2.10.4]{Kollarsings}. So the stack
$(X_1,cD_1)$ is lc if $cD_{1}$ has multiplicity at most 1
at each point. By the previous paragraph,
it suffices to show that $c(a_0\cdots a_{n-2})\deg(D_1)\leq 1$.
This holds by the definition of $c$. So $(X_1,\Delta_1)$ is lc
and hence $(X,\Delta)$ is lc.
\end{proof}

The proof of Corollary \ref{weightedmult} works by bounding
the unweighted multiplicity of $D$
in $A^{n+1}$. At several points in this paper,
it works better to bound a {\it weighted }multiplicity of $D$
at the worst point of $X$,
where information would be lost by going through Theorem \ref{jkbound}.
The idea is that $D$ is given to us as a subspace of a weighted
projective space; so we should use those weights
in analyzing the singularities of $D$, as follows.

\begin{lemma}
\label{someweights}
Let $a_0,\ldots,a_{n+1}$ be positive integers (in any order).
Let $S$ be an irreducible closed substack of a weighted
projective stack $Y=\P^{n+1}(a_0,\ldots,a_{n+1})$.
In coordinates $x_{n+1}=1$, $S$ corresponds to a subvariety of $A^{n+1}$.
Consider the weights $a_0,\ldots,a_n$ on $A^{n+1}$.
Then the weighted multiplicity of $S$ at the origin in $A^{n+1}$
satisfies
$$\mult_a(S)\leq a_{n+1}\deg(S).$$
\end{lemma}

\begin{proof}
Consider the family of hypersurfaces $\{ x_{n+1}=t\} $ in $A^{n+2}$
as $t$ varies. Write $C(S)$ for the affine cone over $S$
in $A^{n+2}$. Consider the weighted multiplicity of
$C(S)\cap\{ x_{n+1}=t\}$ at the point $(0,\ldots,0,t)$.
For $t\neq 0$, this is equal to $\mult_a(S)$. For $t=0$,
this is equal to $\deg(S\cap \{ x_{n+1}=0\})=a_{n+1}\deg(S)$
if $S$ is not contained in the hyperplane
$\{ x_{n+1}=0\}$. Then upper semicontinuity
of the weighted multiplicity gives that $\mult_a(S)\leq
a_{n+1}\deg(S)$. If $S$ is contained in the hyperplane
$\{ x_{n+1}=0\}$, then $S$ does not contain the point
$[0,\ldots,0,1]$; so $\mult_a(S)=0$ and the inequality
again holds.
\end{proof}

\section{The log canonical threshold for a certain singular hypersurface}

In order to show that the klt Fano varieties in Theorem \ref{fano}
are exceptional, we need to analyze their worst singular point,
as follows. Specifically, we need to estimate the log canonical
threshold of any divisor on this singular hypersurface. The proof
involves an induction on these singularities
of different dimensions.

\begin{lemma}
\label{singularity}
Let $n$ be a positive integer.
Let $X_n$ be the hypersurface in $A^{n+1}$ defined by
$0=x_0^2+x_1^3+\cdots+x_n^{s_n}+x_1\cdots x_n$.
Let $D$ be an effective $\Q$-Cartier $\Q$-divisor
in $X_n$. Let $c_i=(s_{n+1}-1)/s_i$ for $0\leq i\leq n$.
Write $\mult_c(D)$ for the $c$-weighted multiplicity of $D$
at the origin.
Then the log canonical threshold of $D$ in $X$ near the origin satisfies
$\lct_0(X_n,D)\geq 1/\mult_c(D)$ if $n=1$ and
$$\lct_0(X_n,D)\geq \frac{2}
{s_n^{n-1}(s_n+1)^2(s_n-1)^{n-3}\mult_c(D)}$$
if $n\geq 2$.
\end{lemma}

\begin{proof}
For $n=1$, $X_1$ is a smooth curve, and the origin in $X_1$
has $c$-multiplicity $1$. So every effective
$\Q$-Cartier $\Q$-divisor
$D$ on $X_1$ has $\lct_0(X,D)\geq 1/\mult_c(D)$, as we want.

For any positive integer $n$, write $\lambda_n$ for the constant
in the lemma, so we are trying to show that $\lct_0(X_n,D)
\geq \lambda_n/\mult_c(D)$. (In particular, let $\lambda_1=1$.)
Suppose that $n\geq 2$ and the inequality
holds for $n-1$ in place of $n$. Let $D$ be an effective
$\Q$-Cartier $\Q$-divisor, and let $c_i=(s_{n+1}-1)/s_i$
for $0\leq i\leq n$. We have to show that if $\mult_c(D)=1$,
then $(X_n,\lambda_nD)$ is lc near the origin.

Consider the modified weights $w_i=c_i$ for $0\leq i\leq n-1$
and $w_n=s_n(s_n+1)/2$. Since $w_n>c_n$, we have
$\mult_w(D)\leq \mult_c(D)=1$. To simplify the numbering,
let $b_i=w_i/s_n$ for all $0\leq i\leq n$. Since $D$ has dimension $n-1$,
we have $\mult_b(D)=s_n^{n-1}\mult_w(D)\leq s_n^{n-1}$.
Here $b_i=(s_n-1)/s_i$ for $0\leq i\leq n-1$
and $b_n=(s_n+1)/2$. Note that $b_n>b_0>\cdots >b_{n-1}$,
contrary to our usual ordering. Finally, let $e=s_n-1$.

The reason for considering the
weights $b_0,\ldots,b_n$ on $A^{n+1}$ is that the weighted
tangent cone of the hypersurface $X_n$ at the origin is klt;
namely, it is the hypersurface $X_{n-1}$ in $Y:=\P^n(b_0,\ldots,b_n)$
of degree $e$
defined by $0=x_0^2+x_1^3+\cdots+x_{n-1}^{s_{n-1}}+x_1\cdots x_n$.
(Here only the monomial $x_n^{s_n}$ has disappeared.)
The stack $X_{n-1}$ is smooth outside the point
$[x_0,\ldots,x_n]=[0,\ldots,0,1]$. The singularity at that point,
in coordinates $x_n=1$, is
$$X_{n-1}=\{0=x_0^2+\cdots+x_{n-1}^{s_{n-1}}+x_1\cdots x_{n-1}\}
\subset A^n,$$
in agreement with the notation of this lemma.

Let $D_{n-1}\subset X_{n-1}$ (inside $Y$) be the $b$-weighted
tangent cone of $D$. Since $D$ is $\Q$-Cartier,
$D_{n-1}$ is $\Q$-linearly equivalent to a rational
multiple of $O_{X_{n-1}}(1)$.
We have $\deg(D_{n-1})=\mult_b(D)\leq s_n^{n-1}$,
as shown above. By Lemma \ref{someweights},
it follows that in coordinates $\{x_n=1\}\cong A^n$,
and with weights $b_0,\ldots,b_{n-1}$ on $A^n$, we have
\begin{align*}
\mult_b(D_{n-1})&\leq b_n\deg(D_{n-1})\\
&\leq s_n^{n-1}(s_n+1)/2,
\end{align*}
using that $b_n=(s_n+1)/2$.
Here $b_i=(s_n-1)/s_i$ for $0\leq i\leq n-1$,
and so our inductive assumption gives that
$\lct_0(X_{n-1},D_{n-1})\geq \lambda_{n-1}/\mult_b(D_{n-1})
\geq 2\lambda_{n-1}/(s_n^{n-1}(s_n+1))$. This number
is at least $\lambda_n$. Indeed, $\lambda_n\sim 2/s_n^{2n-2}$,
and so $\lambda_n$ divided by $2\lambda_{n-1}/(s_n^{n-1}(s_n+1))$
is asymptotic to $1/2$. So it is less than 1 for large $n$,
and a bit more calculation shows that it is less than 1
for all $n\geq 2$. (For $n=2$, using that $\lambda_1=1$,
this ratio is $3/4$, and for $n=3$ this ratio
is $28/33$.) Thus we have shown that
the pair $(X_{n-1},\lambda_nD_{n-1})$
is lc near the point $[x_0,\ldots,x_n]=[0,\ldots,0,1]$.

Since the stack $X_{n-1}$ is smooth outside that point,
we check easily that the pair $(X_{n-1},\lambda_nD_{n-1})$
is lc on all of $X_{n-1}$.
Namely, by Theorem \ref{jkbound}, $D_{n-1}$ has (unweighted)
multiplicity at every point at most $b_0\cdots b_{n-3}b_n\deg(D_{n-1})
\leq b_0\cdots b_{n-3}b_ns_n^{n-1}$. Here $b_0\cdots b_{n-3}
=(s_n-1)^{n-2}/(s_{n-2}-1)\sim s_{n-2}^{4n-9}$,
$b_n=(s_n+1)/2\sim s_{n-1}^4/2$,
and $\lambda_n\sim 2/(s_{n-2}^{8n-8})$.
So $u_n:=\lambda_nb_0\cdots b_{n-3}b_n
s_n^{n-1}\sim 1/s_{n-2}$.
This is less than 1 for $n$ large. With a bit more calculation,
we have $u_n\leq 1$ for all $n\geq 2$ (for example, $u_2=3/4$).
So $\lambda_nD_{n-1}$ has multiplicity at most 1
at every point, for each $n\geq 2$.
Therefore, the pair $(X_{n-1},\lambda_nD_{n-1})$
is lc at points other than $[0,\ldots,0,1]$ as well as
at that point.

By Lemma \ref{lc}, the pair $(X_n,\lambda_nD)$ is lc near the origin
if $(X_{n-1},\lambda_nD_{n-1})$ is lc (as we have shown)
and $K_{X_{n-1}}+\lambda_nD_{n-1}\sim_{\Q} O_{X_{n-1}}(r)$ with $r\leq 0$.
Let us show that $r\leq 0$.
By the adjunction formula, we have $-K_{X_{n-1}}=O_{X_{n-1}}(-e
+\sum_j b_j)=O_X((s_n-1)/2)$.
To show that $r\leq 0$, it is equivalent to show that
$\lambda_n\deg(D_{n-1})=\lambda_nD_{n-1}\cdot c_1(O(1))^{n-2}
\leq (-K_{X_1})\cdot c_1(O(1))^{n-2}$.
Here $(-K_{X_1})\cdot c_1(O(1))^{n-2}=(s_n-1)e/(2b_0\cdots b_n)
=1/((s_n-1)^{n-3}(s_n+1))$.
As a result, the inequality above holds if
$\lambda_n(s_n-1)^{n-3}(s_n+1)\deg(D_{n-1})\leq 1$.
We have $\deg(D_{n-1})=\mult_b(D)\leq s_n^{n-1}(s_n+1)/2$.
So the inequality above holds if
$\lambda_ns_n^{n-1}(s_n+1)^2(s_n-1)^{n-3}/2\leq 1$.
In fact, $\lambda_n$ has been chosen to make equality hold here.
Therefore, the pair $(X_n,
\lambda_nD)$ is lc
near the origin, as we want.
\end{proof}

\section{Exceptionality of the klt Fano example in even dimensions}

In order to show that the klt Fano variety $X$ in Theorem \ref{fano}
is exceptional, it suffices to show that its global log canonical
threshold is greater than 1. It turns out that we can compute
$\glct(X)$ exactly, and
it is doubly exponentially large in terms of $n=\dim(X)$.
In this section, we consider
the case where $n$ is even, which turns out to be simpler.

\begin{theorem}
\label{evenex}
For each even number $n$ at least 2, the klt Fano $n$-fold $X$
in Theorem \ref{fano} is exceptional.
More strongly, the global log canonical threshold of $X$ is equal
to $(s_n-2)a_{n+1}/(s_n-1)\sim s_n^2/4$.
In particular, this is greater than 1.
\end{theorem}

\begin{proof}
To recall the example: let $a_{n+1}=(s_n^2-s_n+2)/4$,
$a_n=(s_n-2)a_{n+1}-(s_n-1)$, $x=-1+a_n+a_{n+1}$,
$d=(s_n-1)x$, and $a_i=d/s_i$ for $0\leq i\leq n-1$.
Then $X$ is a hypersurface of degree $d$ in
$Y:=\P^{n+1}(a_0,\ldots,a_{n+1})$. Since $-d+\sum_j a_j=1$,
we have $-K_X=O_X(1)$. Let $\sigma_n=(s_n-2)a_{n+1}/(s_n-1)$.
We have to show that for every effective $\Q$-divisor $D\sim_{\Q}-K_X$,
the pair $(X,\sigma_nD)$ is lc,
and that this bound is optimal.
As a cycle on the stack $Y$,
$D$ has degree $d/(a_0\cdots a_{n+1})$.

To see that the bound $\sigma_n$ is optimal, let $E$ be the hyperplane
section $X\cap\{ x_{n+1}=0\}$; then we can take $D=(1/a_{n+1})E$.
I claim that $(X,\sigma_nD)$ is not klt, or equivalently
that $(X,((s_n-2)/(s_n-1))E)$ is not klt. 
Here $E$ is given
by the equations $\{ 0=x_{n+1},\; 0=x_0^2+\cdots+x_{n-1}^{s_{n-1}}
\}$. We read off
that $E$ is irreducible.
The stack $E$ is singular at the point $[x_0,\ldots,x_n,x_{n+1}]
=[0,\ldots,0,1,0]$, where $X$ is smooth. In coordinates $x_{n}=1$,
the singularity of $E$
is \'etale-locally isomorphic to the Fermat-type hypersurface singularity
$0=x_0^2+\cdots+x_{n-1}^{s_{n-1}}$ in $A^n$,
which is known to have log canonical threshold equal to
$\min(1,1/2+\cdots+1/s_{n-1})=(s_n-2)/(s_n-1)$
\cite[Example 8.15]{Kollarpairs}. That is, $(X,((s_n-2)/(s_n-1))E)$
is lc but not klt at the point $[0,\ldots,0,1,0]$, as we want.

It remains to show that the pair $(X,\sigma_nD)$ is lc
for every effective $\Q$-divisor $D\sim_{\Q}-K_X$.
The discrepancy of a $\Q$-Cartier
$\Q$-divisor $D$ on $X$ at a given irreducible
divisor over $X$ is an affine-linear function of $D$
\cite[Lemma 2.5]{Kollarsings}. By considering
a log resolution of $(X,D_1+D_2)$, it follows
that if the pairs $(X,D_1)$ and $(X,D_2)$ are lc, then so is
$(X,(1-t)D_1+tD_2)$ for any $t\in [0,1]$.
Since we have already handled the case where $D$ is
$1/a_{n+1}$ times the irreducible divisor $E$,
it suffices to show
that $(X,\sigma_nD)$ is lc when $D\sim_{\Q}-K_X$ and the support
of $D$ does not contain $E$.

In this case, no irreducible component of $D$
is contained in the hyperplane $x_{n+1}=0$.
So Theorem \ref{jkbound} gives that $D$ has multiplicity
at every point
at most $a_0\cdots a_{n-2}a_{n+1}\deg(D)=d/(a_{n-1}a_n)
=s_{n-1}/a_n$. Therefore, $\sigma_nD$ has multiplicity
at every point at most $s_{n-1}(s_n-2)a_{n+1}/((s_n-1)a_n)\sim 1/s_{n-1}$.
This is less than 1. So $(X,\sigma_nD)$ is lc at all smooth
points of the stack $X$ \cite[Claim 2.10.4]{Kollarsings},
hence at all points other than
$[x_0,\ldots,x_n,x_{n+1}]=[0,\ldots,0,1]$. In order
to handle that point, we will switch to analyzing
a certain weighted multiplicity of $D$.

In coordinates $x_{n+1}=1$,
$X$ becomes the hypersurface 
$0=x_0^2+x_1^3+\cdots
+x_{n-1}^{s_{n-1}}+x_n^{s_n}+x_1\cdots x_n$
in $A^{n+1}$. We want to show
that $(X,\sigma_nD)$ is lc near the origin.
Consider the weights $a_0,\ldots,a_n$ on $A^{n+1}$.
Then the $a$-weighted multiplicity of $D$ at the origin in $A^{n+1}$
satisfies
\begin{align*}
\mult_a(D)&\leq a_{n+1}\deg(D)\\
&=d/(a_0\cdots a_n),
\end{align*}
by Lemma \ref{someweights}.

Consider the modified weights $w_i=a_i=d/s_i$
for $0\leq i\leq n-1$, and $w_n=d/s_n$. Here $w_n$ is not an integer,
but we can still define multiplicity for positive rational weights
by scaling.
We have $w_n=(s_n-1)x/s_n=x-(x/s_n)$. Also,
$x=-1+a_n+a_{n+1}=(s_n-1)a_{n+1}-s_n<s_n(a_{n+1}-1)$.
So $x/s_n<a_{n+1}-1$, and hence $w_n=x-(x/s_n)>x-(a_{n+1}-1)=a_n$.
It follows that $\mult_w(D)\leq \mult_a(D)\leq d/(a_0\cdots a_n)$.

Next, let $c_i=s_nw_i/x$ for all $0\leq i\leq n$;
then $c_i=(s_{n+1}-1)/s_i$ for all $0\leq i\leq n$.
Since $D$ has dimension $n-1$, we have
$\mult_c(D)=(x/s_n)^{n-1}\mult_w(D)\leq x^{n-1}d/(s_n^{n-1}a_0
\cdots a_n)=1/(s_n^{n-1}(s_n-1)^{n-2}a_n)$.
By Lemma \ref{singularity}, it follows
that $(X,e_nD)$ is lc near the point $[0,\ldots,0,1]$, where we let
\begin{align*}
e_n&= \frac{2}{s_n^{n-1}(s_n+1)^2(s_n-1)^{n-3}}\cdot
s_n^{n-1}(s_n-1)^{n-2}a_n\\
&=\frac{2(s_n-1)a_n}{(s_n+1)^2}.
\end{align*}

It remains to show that this number $e_n$ is at least $\sigma_n$,
for every even number $n\geq 2$. The ratio $e_n/\sigma_n$
is $2(s_n-1)^2a_n/((s_n+1)^2(s_n-2)a_{n+1})$. Since
$a_{n+1}\sim s_n^2/4$ and $a_n\sim s_n^3/4$, this ratio
is asymptotic to 2, so it is greater than 1 for $n$ large.
With a bit more calculation, we find that $e_n/\sigma_n>1$ for every
even number $n\geq 2$. (For example, $e_2/\sigma_2=441/440$.)
That completes the proof
that $\glct(X)=\sigma_n$. Since this is greater than 1,
the klt Fano variety $X$ is exceptional.
\end{proof}

\section{Exceptionality of the klt Fano example in odd dimensions}
\label{oddexsect}

We now prove the exceptionality of our klt Fano example
in odd dimensions.

\begin{theorem}
\label{oddex}
For each odd number $n$ at least 3, the klt Fano $n$-fold $X$
in Theorem \ref{fano} is exceptional.
More strongly, the global log canonical threshold of $X$
is equal to $(s_n-2)a_{n+1}/(s_n-1)\sim s_n^2/4$.
In particular, this is greater than 1.
\end{theorem}

\begin{proof}
To recall the example: let $a_{n+1}=(s_n^2-3s_n+4)/4$,
$a_n=(s_n-2)a_{n+1}-(s_n-1)$, $x=-1+a_n+a_{n+1}$,
$d=(s_n-1)x$, and $a_i=d/s_i$ for $0\leq i\leq n-1$.
Then $X$ is a hypersurface of degree $d$ in
$Y:=\P^{n+1}(a_0,\ldots,a_{n+1})$. Since $-d+\sum_j a_j=1$,
we have $-K_X=O_X(1)$. Let $\sigma_n=(s_n-2)a_{n+1}/(s_n-1)$.
We have to show that for every effective $\Q$-divisor $D\sim_{\Q}-K_X$,
the pair $(X,\sigma_nD)$ is lc,
and that this bound is optimal. As a cycle on the stack $Y$,
$D$ has degree $d/(a_0\cdots a_{n+1})$.

To see that the bound $\sigma_n$ is optimal, let $E$ be the hyperplane
section $X\cap\{ x_{n+1}=0\}$; then we can take $D=(1/a_{n+1})E$.
I claim that $(X,\sigma_nD)$ is not klt, or equivalently
that $(X,((s_n-2)/(s_n-1))E)$ is not klt. Here $E$ is given
by the equations $\{0=x_{n+1},\; 0=x_0^2+\cdots+x_{n-1}^{s_{n-1}}\}$,
which in particular shows that $E$ is irreducible.
The stack $E$ is singular at the point $[x_0,\ldots,x_n,x_{n+1}]
=[0,\ldots,0,1,0]$, where $X$ is smooth. In coordinates $x_{n+1}=1$,
the singularity of $E$
is \'etale-locally isomorphic to the Fermat-type hypersurface singularity
$0=x_0^2+\cdots+x_{n-1}^{s_{n-1}}$ in $A^n$,
which is known to have log canonical threshold equal to
$\min(1,1/2+\cdots+1/s_{n-1})=(s_n-2)/(s_n-1)$
\cite[Example 8.15]{Kollarpairs}. That is, $(X,((s_n-2)/(s_n-1))E)$
is lc but not klt at the point $[0,\ldots,0,1,0]$, as we want.

It remains to show that the pair $(X,\sigma_nD)$ is lc
for every effective $\Q$-divisor $D\sim_{\Q}-K_X$.
The discrepancy of a $\Q$-Cartier
$\Q$-divisor $D$ on $X$ at a given irreducible
divisor over $X$ is an affine-linear function of $D$
\cite[Lemma 2.5]{Kollarsings}. By considering
a log resolution of $(X,D_1+D_2)$, it follows
that if the pairs $(X,D_1)$ and $(X,D_2)$ are lc, then so is
$(X,(1-t)D_1+tD_2)$ for any $t\in [0,1]$.
Since we have already handled the case where $D$ is
$1/a_{n+1}$ times the irreducible divisor $E$,
it suffices to show
that $(X,\sigma_nD)$ is lc when $D\sim_{\Q}-K_X$ and the support
of $D$ does not contain $E$.

In this case, no irreducible component of $D$
is contained in the hyperplane $x_{n+1}=0$.
So Theorem \ref{jkbound} gives that $D$ has multiplicity
at every point
at most $a_0\cdots a_{n-2}a_{n+1}\deg(D)=d/(a_{n-1}a_n)
=s_{n-1}/a_n$. Therefore, $\sigma_nD$ has multiplicity
at every point at most $s_{n-1}(s_n-2)a_{n+1}/((s_n-1)a_n)\sim 1/s_{n-1}$.
This is less than 1. So $(X,\sigma_nD)$ is lc at all smooth
points of the stack $X$ \cite[Claim 2.10.4]{Kollarsings},
hence at all points other than
$[x_0,\ldots,x_n,x_{n+1}]=[0,\ldots,0,1]$. In order
to handle that point, we will switch to analyzing
a certain weighted multiplicity of $D$.

In coordinates $x_{n+1}=1$,
$X$ becomes the hypersurface 
$0=x_0^2+x_1^3+\cdots
+x_{n-1}^{s_{n-1}}+x_n^{s_n}+x_1\cdots x_{n-1}x_n^2$
in $A^{n+1}$. We want to show
that $(X,\sigma_nD)$ is lc near the origin.
Consider the weights $a_0,\ldots,a_n$ on $A^{n+1}$.
Then the $a$-weighted multiplicity of $D$ at the origin in $A^{n+1}$
satisfies
\begin{align*}
\mult_a(D)&\leq a_{n+1}\deg(D)\\
&=d/(a_0\cdots a_n),
\end{align*}
by Lemma \ref{someweights}.

Let $w_i=a_i$
for $0\leq i\leq n-1$, and let $w_n=(s_n+1)x/4$.
Since $w_n>a_n$,
we have $\mult_w(D)\leq \mult_a(D)\leq d/(a_0\cdots a_n)$.
To simplify the numbering, let $c_i=w_i/x$ for all $0\leq i\leq n$;
then $\mult_c(D)=x^{n-1}\mult_w(D)
\leq x^{n-1}d/(a_0\cdots a_n)$.
Going back through the definitions, this means
that $c_i=(s_n-1)/s_i$ for $0\leq i\leq n-1$
and $c_n=(s_n+1)/4$. Also, let $e=s_n-1$,
so that $d=(s_n-1)x=ex$.
Then we can rewrite our bound as
$\mult_c(D)\leq e/(c_0\cdots c_{n-1}a_n)$.
(Note that the largest numbers from $c_0,\ldots,c_n$
are $c_0,c_1,c_n$, contrary to our usual convention.)

The reason for choosing the weights $c_0,\ldots,c_n$
is that with these weights, the weighted tangent cone at the origin
to the hypersurface $X\subset A^{n+1}$ is klt:
it is the hypersurface $S\subset \P^n(c_0,\ldots,c_n)$ of degree $e$
defined by
$$0=x_0^2+x_1^3+\cdots
+x_{n-1}^{s_{n-1}}+x_1\cdots x_{n-1}x_n^2.$$
The stack $S$ is smooth outside the point
$[x_0,\ldots,x_n]=[0,\ldots,0,1]$. The singularity
of $S$ at that point, in coordinates $x_n=1$, is
$0=x_0^2+x_1^3+\cdots+x_{n-1}^{s_{n-1}}+x_1\cdots x_{n-1}$.
This is exactly the singularity of dimension $n-1$
in Lemma \ref{singularity}.

Let $F$ be the weighted tangent cone of $D$ at the origin
in $A^{n+1}$,
so that $F$ is an effective $\Q$-divisor
in $S\subset \P^n(c_0,\ldots,c_n)$. Since $D$
is $\Q$-Cartier, $F$ is $\Q$-linearly equivalent
to a rational multiple of $O_{S}(1)$.
By inversion of adjunction as in the proof
of Lemma \ref{lc}, the pair $(X,\sigma_nD)$ is lc
if $(S,\sigma_nF)$ is lc
and $K_{S}+\sigma_nF\sim_{\Q} O_{S}(r)$ with $r\leq 0$.

We first check that $r\leq 0$.
By the adjunction formula, we have $-K_{S}=O_{S}(-e
+\sum_j c_j)=O_X((s_n-3)/4)$.
To show that $r\leq 0$, it is equivalent to show that
$\sigma_n\deg(F)=\sigma_nF\cdot c_1(O(1))^{n-2}
\leq (-K_{S})\cdot c_1(O(1))^{n-2}$.
Here $(-K_{S})\cdot c_1(O(1))^{n-2}=(s_n-3)e/(4c_0\cdots c_n)$.
As a result, the inequality above holds if
$\sigma_n\leq (s_n-3)e/(4c_0\cdots c_n\deg(F))$.
We have $\deg(F)=\mult_c(D)\leq e/(c_0\cdots c_{n-1}a_n)$.
So the inequality above holds if
$\sigma_n\leq g_n:=(s_n-3)a_n/(4c_n)$. Here $a_n\sim s_n^3/4$
and $c_n\sim s_n/4$, and so $g_n\sim s_n^3$,
whereas $\sigma_n\sim s_n^2/4$.
With a bit more calculation, we see that
$g_n$ is greater than $\sigma_n$ for each odd $n\geq 3$.
(For example, $g_3\doteq 16026.4$ and $\sigma_3\doteq 420.7$.)

It remains to show that $(S,\sigma_nF)$ is lc.
Since the stack $S$ is smooth outside the point
$[x_0,\ldots,x_n]=[0,\ldots,0,1]$, it is easy to show
that $(S,\sigma_nF)$ is lc outside that point.
Namely, $F$ has degree at most $e/(c_0\cdots c_{n-1}a_n)$.
By Johnson-Koll\'ar's bound (Theorem \ref{jkbound}),
in every orbifold chart $\{ x_i\neq 0\}$ and at every point,
$F$ has multiplicity at most
$c_0\cdots c_{n-3}c_n\deg(F)$ (if $n\geq 5$)
or $c_0c_1\deg(F)$ (if $n=3$).
So $F$ has multiplicity at every point at most
$ec_n/(c_{n-2}c_{n-1}a_n)$ if $n\geq 5$,
or $e/(c_2a_n)$ if $n=3$. 
So $\sigma_n F$ has multiplicity at every point at most
$\sigma_nec_n/(c_{n-2}c_{n-1}a_n)$ if $n\geq 5$,
or $\sigma_3e/(c_2a_n)\doteq 0.17$ if $n=3$. In particular,
this is less than 1 for $n=3$.

For $n\geq 5$, we have
$e\sim s_n\sim s_{n-2}^4$,
$c_{n-2}\sim s_n/s_{n-2}\sim s_{n-2}^3$,
$c_{n-1}\sim s_n/s_{n-1}\sim s_{n-2}^2$,
$c_n\sim s_n/4\sim s_{n-2}^4/4$,
$a_n\sim s_n^3/4\sim s_{n-2}^{12}/4$, and $\sigma_n\sim s_n^2/4
\sim s_{n-2}^8/4$.
So $e\sigma_nc_n/(c_{n-2}c_{n-1}a_n)\sim 1/(4s_{n-2})$,
which is less than 1 for $n$ large. With a bit more calculation,
it is less than 1 for every odd number $n\geq 5$. (For example,
for $n=5$, it is about $0.024$.)
So, for each odd number $n\geq 3$,
$\sigma_nF$ has multiplicity less than 1 everywhere.
Since the stack $S$ is smooth outside the point
$[x_0,\ldots,x_n]=[0,\ldots,0,1]$, it follows
that $(S,\sigma_nF)$ is lc outside that point.

In coordinates $x_n=1$, $F$ corresponds to a codimension-2 cycle
on $A^n$. Using weights $c_0,\ldots,c_{n-1}$ on $A^n$,
Lemma \ref{someweights} gives that the weighted 
multiplicity of $F$ at the origin in $A^{n}$
satisfies
\begin{align*}
\mult_c(F)&\leq c_{n}\deg(F)\\
&\leq ec_n/(c_0\cdots c_{n-1}a_n)\\
&=\frac{s_n+1}{4(s_n-1)^{n-2}a_n}.
\end{align*}

The weights $c_0,\ldots,c_{n-1}$ are those considered
in Lemma \ref{singularity} to analyze the hypersurface
$X_{n-1}\subset A^n$. That lemma gives that
$(S,\sigma_nF)$ is lc near the point $[x_0,\ldots,x_n]
=[0\ldots,0,1]$ if
$$\sigma_n\leq 
\frac{2}{s_{n-1}^{n-2}(s_{n-1}+1)^2(s_{n-1}-1)^{n-4}\mult_c(F)},$$
hence if 
$$\sigma_n\leq \frac{8(s_n-1)^{n-2}a_n}{s_{n-1}^{n-2}(s_{n-1}+1)^2(s_{n-1}-1)^{n-4}(s_n+1)}.$$

I claim that this fraction $f_n$ is greater than $\sigma_n$
for every odd number $n\geq 3$. We have $a_n\sim s_n^3/4$,
and so $f_n\sim s_n^2$,
whereas $\sigma_n\sim s_n^2/4$.
In particular, $f_n>\sigma_n$ for $n$ sufficiently large.
With a bit more calculation, we find that $f_n>\sigma_n$
for all odd $n\geq 3$. (For example, $f_3\doteq 1803.0$
and $\sigma_3\doteq 420.7$.)
That completes the proof
that $\glct(X)=\sigma_n$. Since this is greater than 1,
the klt Fano variety $X$ is exceptional.
\end{proof}

\section{Klt Fano varieties with large bottom weight}

The {\it bottom weight }of a Fano variety $X$
means the smallest positive integer $m$
such that $H^0(X,-mK_X)\neq 0$. The following klt Fano variety
has the largest known bottom weight in even dimensions $n$ at least 4,
asymptotic to $\frac{5}{9}s_n^2$ (hence, crudely, about $2^{2^n}$).
We know that there is some
upper bound for the bottom weight of klt Fano varieties in each dimension,
by Birkar's theorem on boundedness of complements
\cite[Theorem 1.1]{Birkarcomp}.

In particular, Theorem \ref{fanoglct} beats the examples
by Wang and me of klt Fano varieties with large bottom weight
\cite[Theorem 5.1]{TW}. The example here is not quasi-smooth. 

\begin{theorem}
\label{fanoglct}
For each even integer $n$ at least 4, let 
$a_{n+1}=\frac{1}{36}(20s_n^2-295s_n+113)$ and
$a_n=\frac{1}{36}(20s_n^2-55s_n+17)$.
Let $x=-1+a_n+a_{n+1}$,
$d=(s_n-1)x$, and $a_i=d/s_i$ for $0\leq i\leq n-1$.
Then a general hypersurface $X$ of degree $d$ in
$\P^{n+1}(a_0,\ldots,a_{n+1})$ is well-formed and a klt Fano variety,
with $-K_X=O_X(1)$. Its bottom weight is $a_{n+1}$, which is asymptotic
to $\frac{5}{9}s_n^2$.
\end{theorem}

Moraga conjectured that every Fano type variety of dimension $n$
has an $N$-complement for some $N\leq (2s_n-3)(s_n-1)$
\cite[Conjecture 4.1]{Moraga}. Theorem \ref{fanoglct}
implies that this bound would be optimal up to a constant
factor.

It would be interesting
to compute the global log canonical threshold for these examples.
The global log canonical threshold
of a Fano variety is at most the bottom weight, and it seems to be close
to the bottom weight when the bottom weight is large.

\begin{question}
\label{glctfano}
For each even number $n$ at least 4,
does the variety in Theorem \ref{fanoglct}
have the largest bottom weight and the largest glct
among all klt Fano $n$-folds?
\end{question}

I speculate that the optimal examples in odd dimensions
will also have bottom weight and glct asymptotic to $\frac{5}{9}s_n^2$.

The best examples I know in low dimensions are as follows.
I conjecture that the klt del Pezzo surface with largest bottom weight
and largest glct
is $X_{154}\subset \P^3(77,45,19,14)$, for which $\glct(X)=21/2=10.5$;
this is a non-quasi-smooth hypersurface, apparently new.
(The first known klt del Pezzo surface with bottom weight 14
was Kim-Park's quasi-smooth complete intersection
$X_{64,70}\subset \P^4(45,32,25,19,14)$,
which has glct equal to $28/3\doteq 9.33$ \cite[Table 2]{KP}.)

I conjecture that the klt Fano 3-fold with largest bottom weight
and largest glct is another non-quasi-smooth hypersurface,
introduced here:
$$X_{65418}\subset \P^4(32709,21806,9233,884,787),$$
with equation $0=x_0^2+x_1^3+x_2^7x_4+x_1x_2x_3^{38}x_4
+x_3x_4^{82}$.
The largest previously known bottom weight
of a klt Fano 3-fold
occurs for Johnson-Koll\'ar's
quasi-smooth hypersurface
$$X_{37584}\subset
\P^4(18792,12528,5311,547,407)$$
\cite[Introduction]{JK4}.

Finally, the klt Fano 4-fold in Theorem \ref{fanoglct}
has bottom weight 1799223.
The largest previously known bottom weight
of a klt Fano 4-fold is
1094225, which occurs for the quasi-smooth
hypersurface \cite[ID 1228436]{GRD}.

\begin{proof}
(Theorem \ref{fanoglct}) By the properties of the Sylvester sequence,
we have $d-\sum a_i=-1$.
Also, we compute that
the equation of $X$ includes at least the monomials
$0=x_0^2+x_1^3+\cdots
+x_{n-1}^{s_{n-1}}+x_1x_n^bx_{n+1}+x_2\cdots x_{n-1}x_n^{12}x_{n+1}^c$,
where $b=(4s_n-31)/3$ and $c=(5s_n-5)/3$.

We first check that $a_{n+1}$ and $a_n$ are integers. The denominator
36 factors as $2^23^2$.
Since $n$ is even and at least 2, we have $s_n\equiv -1\pmod{8}$
by induction from the definition of the Sylvester sequence,
as in the proof of Theorem \ref{ample}.
So $20s_n^2-295s_n+113\equiv 20(-1)^2-295(-1)+113\equiv 4\pmod{8}$.
It follows that $a_{n+1}$ is integral at 2 and odd. Let $e=a_n-a_{n+1}
=\frac{1}{3}(20s_n-8)$. We see that $e$ is integral at 2 and even,
and so $a_n$ is integral at 2 and odd. Next, we have $s_n\equiv -2\pmod{9}$
since $n\geq 2$. Write $s_n=9t-2$ for an integer $t$.
Then $a_{n+1}=\frac{1}{36}(20s_n^2-295s_n+113)=\frac{1}{4}(180t^2-375t+87)
\equiv 0\pmod{3}$. Also, $e=\frac{1}{3}(20s_n-8)=60t-16\equiv -1\pmod{3}$.
So $a_{n+1}$ and $a_n$ are integers, both are odd, and
$a_{n+1}\equiv 0\pmod{3}$ while $a_n\equiv -1\pmod{3}$. It is also
immediate from the definitions that $a_{n+1}$ and $a_n$ are nonzero
modulo 5.

Let us show that the weighted projective
space $Y=\P^{n+1}(a_0,\ldots,a_{n+1})$ is well-formed.
That is, we have to show that $\gcd(a_0\ldots,\widehat{a_j},
\ldots,a_{n+1})=1$ for each $j$. Let $x=-1+a_n+a_{n+1}$.
It suffices to show
that $\gcd(a_{n+1},x)=1$, $\gcd(a_n,x)=1$, and $\gcd(a_{n+1},a_n,s_n-1)=1$.

We first show that $\gcd(a_{n+1},x)=1$.
Let $p$ be a prime number that divides both $a_{n+1}$ and $x$.
We know that $a_{n+1}$ is odd and not a multiple of 5,
and $x\equiv -1+a_n+a_{n+1}\equiv -1-1+0\equiv -2\pmod{3}$;
so we must have $p> 5$. Since $x=-1+2a_{n+1}+e$,
we have $e\equiv 1\pmod{p}$. That is, $\frac{1}{3}(20s_n-8)\equiv 1
\pmod{p}$, and so $20s_n\equiv 11\pmod{p}$. It follows
that $s_n\equiv 11/20\pmod{p}$. So $0\equiv a_{n+1}
\equiv (1/36)(20(11/20)^2-295(11/20)+113)\equiv -6/5\pmod{p}$.
So $p$ is 2 or 3, contradiction.
So $\gcd(a_{n+1},x)=1$.

Next, we show that $\gcd(a_n,x)=1$. Let $p$ be a prime number
that divides $a_n$ and $x$. Since $a_n$ is not a multiple of 2,
3, or 5, we must have $p>5$.
Since $x=-1+a_n+a_{n+1}$, we have $a_{n+1}\equiv 1\pmod{p}$
and hence $e\equiv -1\pmod{p}$. That is, $\frac{1}{3}(20s_n-8)
\equiv -1\pmod{p}$, so $20s_n\equiv 5\pmod{p}$, and hence
$s_n\equiv 1/4\pmod{p}$. So $0\equiv a_n
\equiv (1/36)(20(1/4)^2-55(1/4)+17)
\equiv 1/8\pmod{p}$, a contradiction. So $\gcd(a_n,x)=1$.

Finally, we show that $\gcd(a_n,a_{n+1},s_n-1)=1$ (and in fact
$\gcd(a_n,s_n-1)=1$). Let $p$ be a prime
number that divides $a_n$ and $s_n-1$. In particular,
$p>5$ because $p$ divides $a_n$. We have $s_n\equiv 1\pmod{p}$,
and so $0\equiv a_{n}\equiv (1/36)(20(1)^2-55(1)+17)
\equiv -1/2$, contradiction. This completes the proof
that $Y$ is well-formed.

To show that $X$ is well-formed, it remains to show
that $X$ does not contain any $(n-1)$-dimensional
coordinate linear subspace of $Y$
along which the variety $Y$ is singular (that is, where
the corresponding smooth stack has nontrivial stabilizer group).
Since the equation of $X$ includes the monomials
$x_0^2,x_1^3,\ldots,x_{n-1}^{s_{n-1}}$, $X$ contains
at most one positive-dimensional coordinate linear subspace of $Y$,
the projective line $Z$ given by $0=x_0=\cdots=x_{n-1}$.
Since $n\geq 4$, it follows that $X$ is well-formed.
Also, the hypersurface $X$ is normal, by Serre's criterion,
using that $X$ is quasi-smooth
outside the curve $Z$. Then
adjunction applies: we have $K_X=O_X(d-\sum a_i)=O_X(-1)$,
as we want.

Next, we show that a general hypersurface $X$ of degree $d$ in $Y$
is klt. As we have said,
$X$ is quasi-smooth (and hence klt) outside the projective line $Z$.
It remains to show that $X$ is klt in the open subsets $x_n\neq 0$
and $x_{n+1}\neq 0$.

In coordinates $x_{n}=1$, the equation of $X$ includes the monomials
$0=x_0^2+x_1^3+\cdots
+x_{n-1}^{s_{n-1}}+x_1x_{n+1}+x_2\cdots x_{n-1}x_{n+1}^c$,
where $c=(5s_n-5)/3$. By Theorem \ref{newton},
it suffices to show that the convex hull of the points
$(2,0,\ldots,0)$, $(0,3,0,\ldots,0)$, \ldots,
$(0,\ldots,0,s_{n-1},0)$, $(0,1,0,\ldots,0,1)$,
$(0,0,1,\ldots,1,c)$ in $\R^{n+1}$ contains
a point with all coordinates less than 1. In fact, we only need
two of these points: namely, $(1/3)(2,0,\ldots,0)+
(2/3)(0,1,0,\ldots,0,1)$ has all coordinates less than 1.
Thus $X$ is klt in the open set $x_n\neq 0$.

It remains to analyze $X$ in coordinates $x_{n+1}=1$.
The equation of $X$ includes the monomials
$0=x_0^2+x_1^3+\cdots
+x_{n-1}^{s_{n-1}}+x_1x_n^b+x_2\cdots x_{n-1}x_n^{12}$.
So it suffices to show that the convex hull of the points
$(2,0,\ldots,0)$, $(0,3,0,\ldots,0)$, \ldots,
$(0,\ldots,0,s_{n-1},0)$, $(0,1,0,\ldots,0,b)$,
$(0,0,1,\ldots,1,12)$ in $\R^{n+1}$ contains
a point with all coordinates less than 1. Indeed, the point
$(1/2-\epsilon)(2,0,\ldots,0)
+(1/3-\epsilon)(0,3,0,\ldots,0)+(1/12+3\epsilon)(0,0,7,0,\ldots,0)
+(1/12-\epsilon)(0,0,1,\ldots,1,12)$ has all coordinates less
than 1 if $0<\epsilon <1/60$.
This completes the proof that $X$ is klt.
\end{proof}

\section{Klt varieties with ample canonical class
and large bottom weight}

The {\it bottom weight }of a projective variety $X$ with $K_X$ ample
means the smallest positive integer $m$
such that $H^0(X,mK_X)\neq 0$. The following klt variety
with ample canonical class
has the largest known bottom weight in even dimensions $n$ at least 4,
asymptotic to $\frac{5}{9}s_n^2$ (hence, crudely, about $2^{2^n}$).
(In other words, we are exhibiting
klt varieties with ample canonical class that have
many vanishing plurigenera.) We know that there is some
upper bound for the bottom weight in each dimension,
by Hacon-M\textsuperscript{c}Kernan-Xu \cite[Theorem 1.3]{HMXacc}.

In particular, Theorem \ref{kampleglct} beats the examples
by Wang and me of klt varieties with ample canonical class
and large bottom weight
\cite[Remark 4.2]{TW}. The example here is not quasi-smooth. 

\begin{theorem}
\label{kampleglct}
For each even integer $n$ at least 4, let 
$a_{n+1}=\frac{1}{36}(20s_n^2-415s_n+161)$ and
$a_n=\frac{1}{36}(20s_n^2-175s_n+65)$.
Let $x=1+a_n+a_{n+1}$,
$d=(s_n-1)x$, and $a_i=d/s_i$ for $0\leq i\leq n-1$.
Then a general hypersurface $X$ of degree $d$ in
$\P^{n+1}(a_0,\ldots,a_{n+1})$ is well-formed and a klt variety
with $K_X=O_X(1)$. Its bottom weight is $a_{n+1}$, which is asymptotic
to $\frac{5}{9}s_n^2$.
\end{theorem}

The proof is completely
parallel to that of Theorem \ref{fanoglct} and hence is omitted.
The proof uses that the equation of $X$ includes
the monomials $0=x_0^2+\cdots+x_{n-1}^{s_{n-1}}+x_1x_n^{13}x_{n+1}^b
+x_2\cdots x_{n-1}x_n^cx_{n+1}^7$, where
$b=(4s_n-19)/3$ and $c=(5s_n-50)/3$.
Here $X$ is not quasi-smooth.

The variety
in Theorem \ref{kampleglct} should also have large
{\it global log canonical threshold }$\glct(X):=\lct(X,K_X)$.
(For varieties with ample canonical
class, this invariant was first studied
by J.~Chen, M.~Chen, and C.~Jiang \cite[section 2.5]{CCJ}.)
The global log canonical threshold
of a variety with ample canonical class is at most the bottom weight,
and it seems to be close
to the bottom weight when the bottom weight is large.

\begin{question}
For each even number $n$ at least 4,
does the variety in Theorem \ref{kampleglct}
have the largest bottom weight and the largest glct
among all klt projective $n$-folds with ample canonical class?
\end{question}

I speculate that the optimal examples in odd dimensions
will also have bottom weight asymptotic to $\frac{5}{9}s_n^2$.

The best examples I know in low dimensions are as follows.
I conjecture that the klt surface with ample canonical class
of largest bottom weight
and largest glct
is $X_{182}\subset \P^3(91,53,23,14)$;
this is a non-quasi-smooth hypersurface, apparently new.
By Brown and Kasprzyk's programs,
the largest bottom weight for a quasi-smooth hypersurface
with $K_X=O_X(1)$ is 13, which occurs for $X_{316}\subset
\P^3(158,101,43,13)$ and $X_{159}\subset \P^3(73,43,29,13)$
\cite{GRD}.

I conjecture that the klt 3-fold with ample canonical class
of largest bottom weight
and largest glct is another non-quasi-smooth hypersurface, introduced here:
$$X_{72954}\subset \P^4(36477,24318,10422,943,793),$$
with equation $0=x_0^2+x_1^3+x_2^7+x_0x_3^{37}x_4^2
+x_1x_2x_3x_4^{47}$.
The largest previously known bottom weight
of a klt 3-fold with ample canonical class
occurs for the quasi-smooth hypersurface
$$X_{18174}\subset
\P^4(8854,5889,2457,507,466)$$
\cite{GRD}.

Finally, the klt 4-fold with ample canonical class
in Theorem \ref{kampleglct}
has bottom weight 1793201.
The largest previously known bottom weight
of a klt 4-fold with ample canonical class is
1127113, which occurs for the quasi-smooth
hypersurface \cite[ID 534198]{GRD}.


\small \sc UCLA Mathematics Department, Box 951555,
Los Angeles, CA 90095-1555

totaro@math.ucla.edu
\end{document}